\documentclass[12pt,leqno]{amsart}
\usepackage{amssymb}
\usepackage{euscript}
\usepackage[hypertex]{hyperref}
\usepackage{colordvi}

\title[Tensors and Graphs]%
       {$GL_n$-Invariant tensors and graphs}

\author[M.~Markl]{Martin~MARKL}
\thanks{The author was supported by the grant GA \v CR 201/08/0397 and by
   the Academy of Sciences of the Czech Republic,
   Institutional Research Plan No.~AV0Z10190503}

\address{Mathematical Institute of the Academy, {\v Z}itn{\'a} 25,
         115 67 Prague 1, The Czech Republic}
\email{markl@math.cas.cz}
\keywords{Invariant tensor, general linear group, graph} 
\subjclass{20G05}

\begin{document}
\baselineskip16pt plus 1pt minus .5pt

\bibliographystyle{plain}

\swapnumbers
\newtheorem{theorem}{Theorem}[section]
\newtheorem{corollary}[theorem]{Corollary}
\newtheorem{observation}[theorem]{Observation}
\newtheorem{lemma}[theorem]{Lemma}
\newtheorem{proposition}[theorem]{Proposition}
\newtheorem{problem}[theorem]{Problem}
\newtheorem{conjecture}[theorem]{Conjecture}
\newtheorem{odstavec}{\hskip -0.1mm}[section]
\newtheorem*{principle}{Principle}
\newtheorem*{itt}{Invariant Tensor Theorem}

\theoremstyle{definition}
\newtheorem{example}[theorem]{Example}
\newtheorem{remark}[theorem]{Remark}
\newtheorem{definition}[theorem]{Definition}

\def\Rada#1#2#3{#1_{#2},\dots,#1_{#3}} \def\GLV{{\GLname(V)}}
\def\aRda#1#2#3{#1^{#2},\dots,#1^{#3}} \def\bfk{{\mathbf k}}
\def\pa{\partial} \def\Con{{\it Con\/}} \def\card{{\rm card}}
\def\GLname{{\rm GL\/}}\def\GL#1#2{{\GLname\/}^{(#1)}_{#2}}
\def\NGLname{{\rm NGL\/}}\def\NGL#1#2{{\NGLname\/}^{(#1)}_{#2}}
\def\nglname{{\mathfrak {ngl}\/}}\def\ngl#1#2{{\nglname\/}^{(#1)}_{#2}}
\def\JGLname{{\rm J\GLname\/}}  \def\pg{{\mathbold p}}
\def\Frname{{\it Fr\/}}\def\IN{{\rm In}}\def\OUT{{\rm Ou}}
\def\fS{{\mathfrak S}} \def\fs{{\mathfrak s}}\def\hh{h}
\def\Ant{{\rm Ant\/}}\def\epi{ \twoheadrightarrow}\def\frH{{\mathfrak H}}
\def\fSout{\fS_{\it out}} \def\fSin{\fS_{\it in}}\def\Grbull{\Gr_\bullet}
\def\Fr#1{{\Frname}^{(#1)}} \def\wR{{\widehat R}} \def\wGr{{\widehat{\Gr}}}
\def\fin{{\mathfrak{{I}}}} \def\udelta{{\underline{\delta}}}
\def\Si{{\mathfrak{{I}}}} \def\So{{\mathfrak{{O}}}}
\def\wrR{\widehat{\mathrm R}}\def\uR{{\underline{R}}}
\def\adjust#1{{\raisebox{-#1em}{\rule{0pt}{0pt}}}} \def\adj{\adjust {.4}}
\def\fout{{\mathfrak{ou}}} \def\rR{{\mathrm R}}
\def\bbbR{{\mathbb R}} \def\bbR{\bbbR}
\def\dr#1#2{\frac{\partial #1}{\partial #2}}
\def\Man{{\tt Man\/}} \def\ctverecek{\raisebox{-0pt}{\smbbox}\hskip -4pt}
\def\gF{{\mathfrak F\/}}\def\gG{{\mathfrak G\/}}\def\gO{{\mathfrak O\/}}
\def\gB{{\mathfrak B\/}}\def\sB{{\EuScript B}}
\def\Fib{{\tt Fib\/}} \def\symexp#1#2{{#1^{\odot #2}}}
\def\Lin{{\mbox {\it Lin\/}}}\def\Sym{{\mbox {\it Sym\/}}}
\def\ext{\mbox{\Large$\land$}}\def\bp{{\mathbf p}}
\def\EXT{\mbox{\raisebox{-.6em}{\rule{1pt}{0pt}}%
         \raisebox{-.2em}{\Huge$\land$}}}
\def\LAND{\mbox{\raisebox{-.2em}{\rule{1pt}{0pt}}%
         \raisebox{-.0em}{\Large$\land$}}}
\def\Jo{{\overline{\JGLname}}}\def\uNat{\underline{\mathfrak {Nat}\/}} 
\def\jo{{\overline {\j}}} \def\CCE{C_{\it CE}}
\def\semidirect{\makebox{\hskip .3mm$\times \hskip -.8mm\raisebox{.2mm}%
                {\rule{.17mm}{1.9mm}}\hskip-.75mm$\hskip 2mm}}
\def\Rn{{\bbR^n}} \def\ot{\otimes} \def\zn#1{{(-1)^{#1}}}
\def\id{{1 \!\! 1}} \def\gh{{\mathfrak h}}
\def\sF{{\EuScript F}}\def\sG{{\EuScript G}}\def\sC{{\EuScript C}}
\def\Map{{\mbox {\it Map\/}}} \def\Nat{{\mathfrak {Nat}\/}} 
\def\WFG{W_{\gF,\gG}}  \def\sqot{\hskip -.3em \otimes \hskip -.3em}
\def\Gr{{\EuScript {G}\rm r}}\def\plGr{{\rm pl\widehat{\EuScript {G}\rm r}}}
\def\GrFG{{\Gr}_{\gF,\gG}}   \def\orGr{{\rm or\EuScript {G}\rm r}}
\def\otexp#1#2{#1^{\ot #2}} \def\semGH{G \semidirect H}
\def\sbbox{{\raisebox {.1em}{\rule{.4em}{.4em}} \hskip .1em}}
\def\bbox{{\raisebox {.1em}{\rule{.6em}{.6em}} \hskip .1em}}
\def\mbbox{{\raisebox {.1em}{\rule{.4em}{.4em}} \hskip .1em}}
\def\smbbox{{\raisebox {.0em}{\rule{.5em}{.5em}} \hskip .0em}}
\def\vmbbox{{\raisebox {.1em}{\rule{.2em}{.2em}} \hskip .1em}}
\def\ii{{(\infty)}} \def\phi{\varphi}\def\uGrFG{{\uGr}_{\gF,\gG}} 
\def\rada#1#2{{#1,\ldots,#2}}\def\uGr{{\EuScript {G}\underline{\rm r}}}
\def\dirlim{{{\mathop{{\rm lim}}\limits_{\longrightarrow}}\hskip 1mm}}
\def\dCE{\delta_{\it CE\/}} \def\plR{{\mathrm {pl}\widehat{\EuScript R}}}  
\def\orR{{\mathit orR}} \def\fA{{\widehat{\Sigma}}}
\def\Vert{{\it Vert\/}} \def\bfb{{\mathbf b}}
\def\Ker{{\it Ker}} \def\wc{\circlearrowright}
\def\Lie{{\mathcal L{\it ie\/}}}\def\Im{{\it Im}}
\def\Ext{\mathop{{\rm \EXT}}\displaylimits}
\def\Land{\mathop{{\LAND}}\displaylimits}
\def\skelet{
\begin{picture}(2,4)(0,0)
\put(1,1){\oval(2,2)[t]}
\put(1,-1){\oval(2,2)[b]}
\put(2,0.25){\makebox(0,0)[c]{\vdots}}
\end{picture}
}

\def\anchor{$\unitlength .25cm
\begin{picture}(1,1.4)(-1,-.7)
\put(-.45,.55){\makebox(0,0)[cc]{$\sbbox$}}
\put(-.5,-.8){\vector(0,1){1.2}}
\end{picture}$}

\def\uunit{$\unitlength .25cm
\begin{picture}(1,1.4)(-1,-.9)
\put(-.45,.55){\makebox(0,0)[cc]{$\sbbox$}}
\put(-.45,-1){\makebox(0,0)[cc]{$\bullet$}}
\put(-.5,-.8){\vector(0,1){1.2}}
\end{picture}$}

\def\black{$\unitlength .25cm
\begin{picture}(1,1.4)(-1,-.7)
\put(-.45,-.3){\makebox(0,0)[cc]{\Large$\bullet$}}
\put(-.5,0){\vector(0,1){1.3}}
\end{picture}\hskip .2em $}

\def\white{$\unitlength .25cm
\begin{picture}(1,1.4)(-1,-.7)
\put(-.45,-.3){\makebox(0,0)[cc]{\Large$\circ$}}
\put(-.5,0.1){\vector(0,1){1.2}}
\end{picture}\hskip .2em $}

\def\osG{\overline{\Gr}}    \def\ti{{\times \infty}}      \def\td{{\times d}}
\def\Tr{{\it Tr\/}}         \def\oti{{\otimes \infty}}    
\def\Com{{\EuScript C}{\it om}} \def\calA{{\EuScript A}} 
\def\bfc{{\mathbf c}}       \def\calQ{{\EuScript Q}}
\def\calP{{\EuScript P}}     \def\pLie{p{\mathcal L{\it ie\/}}}
\def\Re{{\mathbb R}}        \def\od{{\otimes d}}
\def\crr{connected replacement rules}
\def\odrazka#1{{\raisebox{#1}{\rule{0pt}{0pt}}}}
\def\ccdot{\mbox {\scriptsize \hskip .3em $\bullet$\hskip .3em}}
\def\Grtr{\Gr_{\bullet\nabla\it Tr}}      \def\vt{\vartheta}
\def\Grd #1#2#3{\Gr^{#1}_{\bullet#2}[#3](d)} \def\Edg{{\it Edg}}
\def\Lab{{\it Lab}}

\def\borelioza#1#2{
\unitlength.7cm
\put(0,-1){
\put(0,-.1){
\put(0,2){\put(0.03,0){\makebox(0,0)[cc]{$\bbox$}}}
\put(0,1){\vector(0,1){.935}}
\put(0,1){\makebox(0,0)[cc]{\Large$\bullet$}}
\put(0,.7){\makebox(0,0)[t]{\scriptsize$#1$}}}
\put(.4,.2){
\put(2.09,.85){\makebox(0,0)[cc]{\oval(1.5,1.5)[b]}}
\put(2.09,1.15){\makebox(0,0)[cc]{\oval(1.5,1.5)[t]}}
\put(2.85,1.22){\line(0,1){.3}}
\put(.7,.7){\vector(1,1){.55}}
\put(.7,.7){\makebox(0,0)[cc]{\Large$\bullet$}}
\put(1.35,1.35){\makebox(0,0)[cc]{\Large$\bullet$}}
\put(1.32,1.25){\makebox(0,0)[tc]{\vector(0,1){0}}}
\put(1,1.45){\makebox(0,0)[r]{\scriptsize $F$}}
\put(.7,0.4){\makebox(0,0)[t]{\scriptsize $#2$}}
}}}
\def\cases#1#2#3#4{
                  \left\{
                         \begin{array}{ll}
                           #1,\ &\mbox{#2}
                           \\
                           #3,\ &\mbox{#4}
                          \end{array}
                   \right.
}
\def\boreliozaInv#1#2{
\unitlength.7cm
\put(0,-1){
\put(0,-.1){
\put(0,2){\put(0.03,0){\makebox(0,0)[cc]{$\bbox$}}}
\put(0,1){\vector(0,1){.935}}
\put(0,1){\makebox(0,0)[cc]{\Large$\bullet$}}
\put(0,.7){\makebox(0,0)[t]{\scriptsize$#1$}}
}
\put(3.8,.2){
\put(-2.09,.85){\makebox(0,0)[cc]{\oval(1.5,1.5)[b]}}
\put(-2.09,1.15){\makebox(0,0)[cc]{\oval(1.5,1.5)[t]}}
\put(-2.85,1.22){\line(0,1){.3}}
\put(-.7,.7){\vector(-1,1){.55}}
\put(-.7,.7){\makebox(0,0)[cc]{\Large$\bullet$}}
\put(-1.35,1.35){\makebox(0,0)[cc]{\Large$\bullet$}}
\put(-1.32,1.25){\makebox(0,0)[tc]{\vector(0,1){0}}}
\put(-1,1.45){\makebox(0,0)[l]{\scriptsize $F$}}
\put(-.7,0.4){\makebox(0,0)[t]{\scriptsize $#2$}}
}}}

\def\sigmadva{{\unitlength.5cm\thicklines
\put(0,-.5){\vector(1,1){1}}
\put(1,-.5){\vector(-1,1){1}}
\put(2,-.5){\vector(0,1){1}}
\put(0,-.7){\makebox(0,0)[t]{\scriptsize 1}}
\put(1,-.7){\makebox(0,0)[t]{\scriptsize 2}}
\put(2,-.7){\makebox(0,0)[t]{\scriptsize 3}}
\put(0,.7){\makebox(0,0)[b]{\scriptsize 1}}
\put(1,.7){\makebox(0,0)[b]{\scriptsize 2}}
\put(2,.7){\makebox(0,0)[b]{\scriptsize 3}}
}}

\def\sigmatri{{\unitlength.5cm\thicklines
\put(0,-.5){\vector(0,1){1}}
\put(1,-.5){\vector(1,1){1}}
\put(2,-.5){\vector(-1,1){1}}
\put(0,-.7){\makebox(0,0)[t]{\scriptsize 1}}
\put(1,-.7){\makebox(0,0)[t]{\scriptsize 2}}
\put(2,-.7){\makebox(0,0)[t]{\scriptsize 3}}
\put(0,.7){\makebox(0,0)[b]{\scriptsize 1}}
\put(1,.7){\makebox(0,0)[b]{\scriptsize 2}}
\put(2,.7){\makebox(0,0)[b]{\scriptsize 3}}
}}

\def\sigmactyri{{\unitlength.5cm\thicklines
\put(0,-.5){\vector(2,1){2}}
\put(1,-.5){\vector(0,1){1}}
\put(2,-.5){\vector(-2,1){2}}
\put(0,-.7){\makebox(0,0)[t]{\scriptsize 1}}
\put(1,-.7){\makebox(0,0)[t]{\scriptsize 2}}
\put(2,-.7){\makebox(0,0)[t]{\scriptsize 3}}
\put(0,.7){\makebox(0,0)[b]{\scriptsize 1}}
\put(1,.7){\makebox(0,0)[b]{\scriptsize 2}}
\put(2,.7){\makebox(0,0)[b]{\scriptsize 3}}
}}

\def\sigmapet{{\unitlength.5cm\thicklines
\put(0,-.5){\vector(1,1){1}}
\put(1,-.5){\vector(1,1){1}}
\put(2,-.5){\vector(-2,1){2}}
\put(0,-.7){\makebox(0,0)[t]{\scriptsize 1}}
\put(1,-.7){\makebox(0,0)[t]{\scriptsize 2}}
\put(2,-.7){\makebox(0,0)[t]{\scriptsize 3}}
\put(0,.7){\makebox(0,0)[b]{\scriptsize 1}}
\put(1,.7){\makebox(0,0)[b]{\scriptsize 2}}
\put(2,.7){\makebox(0,0)[b]{\scriptsize 3}}
}}

\def\sigmasest{{\unitlength.5cm\thicklines
\put(0,-.5){\vector(2,1){2}}
\put(1,-.5){\vector(-1,1){1}}
\put(2,-.5){\vector(-1,1){1}}
\put(0,-.7){\makebox(0,0)[t]{\scriptsize 1}}
\put(1,-.7){\makebox(0,0)[t]{\scriptsize 2}}
\put(2,-.7){\makebox(0,0)[t]{\scriptsize 3}}
\put(0,.7){\makebox(0,0)[b]{\scriptsize 1}}
\put(1,.7){\makebox(0,0)[b]{\scriptsize 2}}
\put(2,.7){\makebox(0,0)[b]{\scriptsize 3}}
}}

\def\brace{
\put(0,0){\line(0,-1){2}}
\put(0,0){\vector(-1,0){.2}}
\put(0,-2){\vector(-1,0){.2}}
}

\begin{abstract}
We describe a correspondence between
$\GLname_n$-invariant tensors and graphs. We then show how this
correspondence accommodates various types of symmetries and orientations.
\end{abstract}

\maketitle

\baselineskip 17pt plus 1pt minus .5pt

\section*{Introduction}

Let $V$ be a finite dimensional vector space over a field $\bfk$ of
characteristic zero and $\GLV$ the group of invertible linear
endomorphisms of $V$.
The classical (Co)Invariant Tensor Theorem recalled in Section~\ref{s1} 
states that the space of
$\GLV$-invariant linear maps between tensor products of copies of $V$
is generated by specific `elementary invariant tensors' and that these
elementary tensors are linearly independent if the dimension of $V$ is
big enough.

We will observe that elementary invariant tensors are in one-to-one
correspondence with contraction schemes for indices which are, in turn,
described by graphs. We then show how this translation between
invariant tensors and linear combination of graphs 
accommodates various types of symmetries and orientations.

The above type of description of invariant tensors by graphs was
systematically used by M.~Kontsevich in his seminal
paper~\cite{kontsevich:93}. Graphs representing tensors 
appeared also in the work of several other authors, let us
mention at least J.~Conant, A.~Hamilton, A.~Lazarev, J.-L.~Loday, S.~Mahajan
M.~Mulase, M.~Penkava, K.~Vogtmann, A.~Schwarz and G.~Weingart.

We were, however, not able to find a~suitable reference containing all
details. The need for such a reference appeared in connection with our
paper~\cite{markl:na} that provided a vocabulary between natural
differential operators and graph complexes. Indeed, this note was
originally designed as an appendix to~\cite{markl:na}, but we believe
that it might be of independent interest. It supplies necessary
details to~\cite{markl:na} and its future applications, and also puts
the `abstract tensor calculus' attributed to R.~Penrose onto a solid
footing.

\noindent 
{\bf Acknowledgement.} We would like to express our thanks to
J.-L.~Loday and J.~Stasheff for useful comments and remarks
concerning the first draft of this note.

\vskip 1cm

\noindent 
{\bf Table of content:} \ref{s1}.  
                 Invariant Tensor Theorem: A recollection -- page~\pageref{s1}
                   \hfill\break\noindent 
\hphantom{{\bf Table of content:\hskip .5mm}}  \ref{s2}.
                     Graphs appear: An example  -- page~\pageref{s2}
\hfill\break\noindent 
\hphantom{{\bf Table of content:\hskip .5mm}}  \ref{s3}.
                     The general case -- page~\pageref{s3}
\hfill\break\noindent 
\hphantom{{\bf Table of content:\hskip .5mm}}  \ref{s4}. 
                    Symmetries occur -- page~\pageref{s4}
                \hfill\break\noindent 
\hphantom{{\bf Table of content:\hskip .5mm}}  \ref{s5}.
                    A particular case  -- page~\pageref{s5}

\section{Invariant Tensor Theorem: A recollection}
\label{s1}

Recall that, for finite-dimensional $\bfk$-vector spaces $U$ and $W$,
one has canonical isomorphisms
\begin{equation}
\label{conon}
\Lin(U,W)^* \cong \Lin(W,U),\
\Lin(U,V) \cong U^* \ot V
\mbox { and }
(U \ot W)^* \cong U^* \ot V^*, 
\end{equation}
where $\Lin(-,-)$ denotes the space of $\bfk$-linear maps, $(-)^*$ the
linear dual and $\ot$ the tensor product over $\bfk$. The first
isomorphism in~(\ref{conon}) is induced by the non-degenerate pairing
\[
\Lin(U,W) \ot \Lin(W,U) \to \bfk
\]
that takes $f \ot g \in \Lin(U,W) \ot \Lin(W,U)$ into the trace of the
composition $\Tr(f \circ g)$, the remaining two isomorphisms are
obvious. In this note, by a {\em canonical isomorphism\/} we will usually
mean a composition of isomorphisms of the above types.
Einstein's convention assuming summation
over repeated (multi)indices is used. We will
also assume that the ground field $\bfk$ is of characteristic zero.

In what follows, $V$ will be an $n$-dimensional $\bfk$-vector space
and $\GLV$  the group of linear
automorphisms of $V$. 
We start by considering 
the vector space $\Lin(\otexp Vk,\otexp Vl)$ of $\bfk$-linear 
maps $f : \otexp Vk \to \otexp Vl$, $k,l \geq 0$. 
Since both $\otexp Vk$ and $\otexp Vl$ are natural
$\GLV$-modules, it makes sense to study the subspace
$\Lin_\GLV(\otexp Vk,\otexp Vl) \subset \Lin(\otexp Vk,\otexp Vl)$ of
$\GLV$-equivariant maps.

As there are no $\GLV$-equivariant maps in $\Lin(\otexp Vk,\otexp
Vl)=0$ if $k \not= l$ (see, for
instance,~\cite[\S24.3]{kolar-michor-slovak}), the only interesting
case is $k=l$. For a permutation $\sigma \in \Sigma_k$,
define the {\em elementary invariant tensor \/} $t_\sigma \in 
\Lin(\otexp Vk,\otexp Vk)$ as the map given by
\begin{equation}
\label{jdu_k_doktorovi}
t_\sigma(v_1 \otimes \cdots \otimes v_k) := 
v_{\sigma^{-1}(1)}\otimes \cdots \otimes v_{\sigma^{-1}(k)},\
\mbox { for }
\Rada v1k \in V.
\end{equation}
It is simple to verify that $t_\sigma$ is $\GLV$-equivariant. The
following theorem is a celebrated result of H.~Weyl~\cite{weyl}.

\begin{itt}
\label{itt}
The space $\Lin_\GLV(\otexp Vk,\otexp Vk)$ is spanned by elementary
invariant tensors $t_\sigma$, $\sigma \in \Sigma_k$. If $\dim(V) \geq
k$, the tensors $\{t_\sigma\}_{\sigma \in \Sigma_k}$ are linearly
independent.
\end{itt}

This form of the Invariant Tensor Theorem is a straightforward translation
of~\cite[Theorem~2.1.4]{fuks} describing
invariant tensors in $\otexp{{V^*}}k \ot \otexp Vk$ and remarks following
this theorem, see also~\cite[Theorem~24.4]{kolar-michor-slovak}. 
The Invariant Tensor Theorem can be reformulated into saying that the map
\begin{equation}
\label{preziji_to?}
{\mathcal R}_n : \bfk[\Sigma_k] \to \Lin_\GLV(\otexp Vk,\otexp Vk)
\end{equation}
from the group ring of $\Sigma_k$ to the subspace of
$\GLV$-equivariant maps given by ${\mathcal R}_n(\sigma) := t_\sigma$, $\sigma
\in \Sigma_k$, is always an epimorphism and is an isomorphism for $n
\geq k$ (recall $n$ denoted the dimension of $V$).

The tensors $\{t_\sigma\}_{\sigma \in
\Sigma_k}$ are not linearly independent if $\dim(V) <k$. 
For a subset $S \subset\{\rada 1k\}$ such that
$\card(S) > \dim(V)$, denote by $\Sigma_S$ the subgroup of 
$\Sigma_k$ consisting of permutations that leave the complement $\{\rada
1k\}\setminus S$ fixed. It is simple to verify that then
\begin{equation}
\label{Pozitri_zpet_do_Prahy}
\sum_{\sigma \in \Sigma_S}{\rm sgn\/}(\sigma) \cdot  t_\sigma  = 0
\end{equation}
in $\Lin_\GLV(\otexp Vk,\otexp Vk)$. By ~\cite[II.1.3]{fuks}, 
all relations between the elementary invariant tensors
are induced by the relations of the above type. In other
words, the kernel of the map ${\mathcal R}_n$ in~(\ref{preziji_to?})
is generated by the expressions 
\[
\sum_{\sigma \in \Sigma_S}{\rm sgn\/}(\sigma) \cdot \sigma \in
 \bfk[\Sigma_k],
\]
where $S$ and $\Sigma_S$ are as above. Observe that, with the
convention used in~(\ref{jdu_k_doktorovi}) involving the inverses of
$\sigma$ in the right hand side, ${\mathcal R}_n$ is a ring
homomorphism.

\begin{definition}
\label{stab}
By the {\em stable range\/} we mean the situation when $\dim(V)
\geq k$, that is, when the map 
${\mathcal R}_n$ in~(\ref{preziji_to?}) is a monomorphism.
\end{definition}

\section{Graphs appear: An example}
\label{s2}

In this section we analyze an example that illustrates how the Invariant
Tensor Theorem leads to graphs.
We are going to describe invariant tensors in 
$\Lin\left(\adj\otexp V2 \sqot \Lin(\otexp V2,V),V\right)$. The canonical 
identifications~(\ref{conon}) determine a
$\GLV$-equivariant isomorphism
\[
\Phi : \Lin\left(\adj\otexp V2 \ot \Lin(\otexp V2,V),V\right) 
\cong \Lin(\otexp V3,\otexp V3).
\]
Applying the Invariant Tensor Theorem to $\Lin(\otexp V3,\otexp V3)$,
one concludes that the subspace $\Lin_\GLV(\otexp V2 \sqot \Lin(\otexp
V2,V),V)$ is spanned by $\Phi^{-1}(t_\sigma)$, $\sigma \in \Sigma_3$,
and that these generators
are linearly independent if $\dim(V) \geq 3$. It is a simple exercise
to calculate the tensors $\Phi^{-1}(t_\sigma)$ explicitly. The results
are shown in the second column of the table in Figure~\ref{table} in
which $X \ot Y \ot F$ is an element of $\otexp V2 \sqot \Lin(\otexp
V2,V)$ and $\Tr(-)$ the trace of a linear map $V \to V$.

\begin{figure}
\unitlength.9cm
\begin{picture}(15,14)(-1,-12)
\put(0,-.5){
\put(1,1){\makebox(0,0)[l]{$\Phi^{-1}(t_\sigma)$:}}
\put(8,1.45){\makebox(0,0)[l]{coordinate}}
\put(8,.9){\makebox(0,0)[l]{form:}}
\put(11,1){\makebox(0,0)[l]{graph:}}
}

\thinlines
\put(-2,0){\line(1,0){16}}
\put(-2,0.1){\line(1,0){16}}
\put(0.6,1){\line(0,-1){12.7}}
\put(0.7,1){\line(0,-1){12.7}}
\put(7.75,1){\line(0,-1){12.7}}
\put(10.6,1){\line(0,-1){12.7}}

\put(-2,-1){\makebox(0,0)[l]{$\sigma = {\it identity}$}}
\put(1,-1){\makebox(0,0)[l]{$X\ot Y \ot F \mapsto F(X,Y)$}}
\put(8,-1){\makebox(0,0)[l]{$X^jY^kF^i_{jk}e_i$}}
\put(11.8,-1.5){
\unitlength.5cm
\put(0,2){\makebox(0,0)[cc]{\hskip .5mm$\bbox$}}
\put(0,1.08){\vector(0,1){.85}}
\put(0,1){\makebox(0,0)[cc]{\Large$\bullet$}}
\put(0.3,1.2){\makebox(0,0)[l]{\scriptsize$F$}}
\put(-1,-1){
\put(0,1){{\vector(1,1){.92}}}
\put(0,1){\makebox(0,0)[cc]{\Large$\bullet$}}
\put(0,0.4){\makebox(0,0)[t]{\scriptsize$X$}}
}
\put(1,-1){
\put(0,1){{\vector(-1,1){.92}}}
\put(0,1){\makebox(0,0)[cc]{\Large$\bullet$}}
\put(0,.4){\makebox(0,0)[t]{\scriptsize$Y$}}
}}

\put(-2,-3){\makebox(0,0)[l]{$\sigma =$}}
\put(-1,-3){\sigmadva}
\put(1,-3){\makebox(0,0)[l]{$X\ot Y \ot F \mapsto F(Y,X)$}}
\put(8,-3){\makebox(0,0)[l]{$X^jY^kF^i_{kj}e_i$}}
\put(11.8,-3.5){
\unitlength.5cm
\put(0,2){\makebox(0,0)[cc]{\hskip .5mm$\bbox$}}
\put(0,1.08){\vector(0,1){.85}}
\put(0,1){\makebox(0,0)[cc]{\Large$\bullet$}}
\put(0.3,1.2){\makebox(0,0)[l]{\scriptsize$F$}}
\put(-1,-1){
\put(0,1){{\vector(1,1){.92}}}
\put(0,1){\makebox(0,0)[cc]{\Large$\bullet$}}
\put(0,0.4){\makebox(0,0)[t]{\scriptsize$Y$}}
}
\put(1,-1){
\put(0,1){{\vector(-1,1){.92}}}
\put(0,1){\makebox(0,0)[cc]{\Large$\bullet$}}
\put(0,.4){\makebox(0,0)[t]{\scriptsize$X$}}
}}

\put(-2,-5){\makebox(0,0)[l]{$\sigma =$}}
\put(-1,-5){\sigmatri}
\put(1,-5){\makebox(0,0)[l]{$X\ot Y \ot F \mapsto Y \otimes \Tr(F(X,-))$}}
\put(8,-5){\makebox(0,0)[l]{$X^jY^iF^k_{jk}e_i$}}
\put(11.1,-5.2){\borelioza YX}

\put(-2,-7){\makebox(0,0)[l]{$\sigma =$}}
\put(-1,-7){\sigmapet}
\put(1,-7){\makebox(0,0)[l]{$X\ot Y \ot F \mapsto Y \otimes \Tr(F(-,X))$}}
\put(8,-7){\makebox(0,0)[l]{$X^jY^iF^k_{kj}e_i$}}
\put(11.1,-7.2){\boreliozaInv YX}

\put(-2,-9){\makebox(0,0)[l]{$\sigma =$}}
\put(-1,-9){\sigmactyri}
\put(1,-9){\makebox(0,0)[l]{$X\ot Y \ot F \mapsto X \otimes \Tr(F(-,Y))$}}
\put(8,-9){\makebox(0,0)[l]{$X^iY^jF^k_{kj}e_i$}}
\put(11.1,-9.2){\boreliozaInv XY}

\put(-2,-11){\makebox(0,0)[l]{$\sigma =$}}
\put(-1,-11){\sigmasest}
\put(1,-11){\makebox(0,0)[l]{$X\ot Y \ot F \mapsto X \otimes \Tr(F(Y,-))$}}
\put(8,-11){\makebox(0,0)[l]{$X^iY^jF^k_{jk}e_i$}}
\put(11.1,-11.2){\borelioza XY}

\put(14.2,-1){\brace}
\put(14.2,-5){\brace}
\put(14.2,-9){\brace}

\end{picture}
\caption{\label{table} Invariant tensors in $\Lin(\otexp V2 \ot
\Lin(\otexp V2,V),V)$. The meaning of vertical braces on the right is
explained in Example~\protect\ref{456}.}
\end{figure}

Let us fix a basis $\{\Rada e1n\}$ of $V$ and write $X = X^ae_a$, $Y =
Y^ae_a$ and $F(e_a,e_b) = F^c_{ab} e_c$, for some scalars $X^a, Y^a, F^c_{ab}
\in \bfk$, $1\leq a,b,c \leq n$. The corresponding 
coordinate forms of the elementary
tensors are shown in the
third column of the table. Observe that the expressions in this
column are all possible {\em contractions of indices\/} of the tensors $X$,
$Y$ and~$F$.

The contraction schemes for indices are encoded by the rightmost
column as follows. Given a graph $G$ from this column, decorate its
edges by symbols $i,j,k$. For example, for the graph in the bottom
right corner of the table, choose the decoration
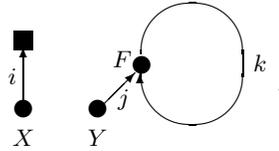
\[
\unitlength.9cm
\begin{picture}(5,2)(-1.5,.7)
\put(0,2){\put(0.03,0){\makebox(0,0)[cc]{$\bbox$}}}
\put(0,1){\vector(0,1){.935}}
\put(0,1){\makebox(0,0)[cc]{\Large$\bullet$}}
\put(0,.7){\makebox(0,0)[t]{\scriptsize$X$}}
\put(.4,.3){
\put(2.09,.85){\makebox(0,0)[cc]{\oval(1.5,1.5)[b]}}
\put(2.09,1.15){\makebox(0,0)[cc]{\oval(1.5,1.5)[t]}}
\put(2.85,1.22){\line(0,1){.3}}
\put(.7,.7){\vector(1,1){.55}}
\put(.7,.7){\makebox(0,0)[cc]{\Large$\bullet$}}
\put(1.35,1.35){\makebox(0,0)[cc]{\Large$\bullet$}}
\put(1.32,1.25){\makebox(0,0)[tc]{\vector(0,1){0}}}
\put(1.2,1.45){\makebox(0,0)[r]{\scriptsize $F$}}
\put(.7,0.4){\makebox(0,0)[t]{\scriptsize $Y$}}
\put(-.5,1.2){\makebox(0,0)[r]{\scriptsize$i$}}
\put(1,1){\makebox(0,0)[lt]{\scriptsize$j$}}
\put(3,1.4){\makebox(0,0)[l]{\scriptsize$k$}}
\put(3.4,1){\makebox(0,0){.}}
}
\end{picture}
\]  
To each vertex of this edge-decorated graph we assign the coordinates of
the corresponding tensors with the names of indices determined by
decorations of edges adjacent to this vertex. For example, to the
$F$-vertex we assign $F^k_{jk}$, because its left ingoing edge is
decorated by $j$ and its right ingoing edge which happens to be the
same as its  outgoing edge, is decorated by $k$. The vertex \anchor,
called {\em the anchor\/}, plays a special role. 
We assign to it the basis of $V$ indexed by the decoration of its
ingoing edge. We get
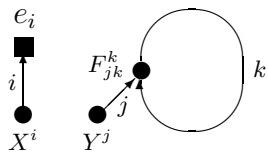
\[
\unitlength.9cm
\begin{picture}(5,2)(-2,.7)
\put(0,2){\put(0.03,0){\makebox(0,0)[cc]{$\bbox$}}}
\put(0,2.3){\put(0.03,0){\makebox(0,0)[b]{$e_i$}}}
\put(0,1){\vector(0,1){.935}}
\put(0,1){\makebox(0,0)[cc]{\Large$\bullet$}}
\put(0,.8){\makebox(0,0)[t]{\scriptsize$X^i$}}
\put(.4,.3){
\put(2.09,.85){\makebox(0,0)[cc]{\oval(1.5,1.5)[b]}}
\put(2.09,1.15){\makebox(0,0)[cc]{\oval(1.5,1.5)[t]}}
\put(2.85,1.22){\line(0,1){.3}}
\put(.7,.7){\vector(1,1){.55}}
\put(.7,.7){\makebox(0,0)[cc]{\Large$\bullet$}}
\put(1.35,1.35){\makebox(0,0)[cc]{\Large$\bullet$}}
\put(1.32,1.25){\makebox(0,0)[tc]{\vector(0,1){0}}}
\put(1.1,1.45){\makebox(0,0)[r]{\scriptsize $F^k_{jk}$}}
\put(.7,0.5){\makebox(0,0)[t]{\scriptsize $Y^j$}}
\put(-.5,1.2){\makebox(0,0)[r]{\scriptsize$i$}}
\put(1,1){\makebox(0,0)[lt]{\scriptsize$j$}}
\put(3,1.4){\makebox(0,0)[l]{\scriptsize$k$}}
}
\end{picture}
\]  
As the final step we take the product of the factors assigned to
vertices and perform the summation over repeated indices. The result
is
\[
\sum_{1 \leq i,j,k \leq n}X^iY^jF^k_{jk}e_i.
\] 
In this formula we made an exception from Einstein's convention and
wrote the summation explicitly to emphasize the idea of the
construction.  A formal general definition of this process of
interpreting graphs as contraction schemes is given below.

Let $\wGr_{\rm ex}$ be the vector space spanned by the six graphs in
the last column of the table; the hat indicates that the graphs are
not oriented. The subscript ``ex'' is an abbreviation of
``example,'' and distinguishes this space from other spaces
with similar names used throughout the note. The procedure described
above gives an epimorphism
\begin{equation}
\label{boli_mne_v_krku}
\wrR_n : \wGr_{\rm ex} \to \Lin_\GLV\left(\adj\otexp V2 \ot  
\Lin(\otexp V2,V),V\right)
\end{equation}
which is an isomorphism if $n \geq 3$. The map $\wR_n$
defined in this way obviously does not depend on the choice of the basis
$\{\Rada e1n\}$ of $V$.

The space $\wGr_{\rm ex}$ can also be defined as the span of all
directed graphs with three unary vertices
\begin{equation}
\label{aaa}
\unitlength .5cm
\begin{picture}(0,1)(0,.2)
\put(-.5,0){\makebox(0,0)[cc]{\Large$\bullet$}}
\put(0,0){\makebox(0,0)[bl]{\scriptsize $X$}}
\put(0.9,0.2){\makebox(0,0){,}}
\put(-.5,0){\vector(0,1){1.2}}
\end{picture}
\hskip 3em
\unitlength .5cm
\begin{picture}(0,1)(0,.2)
\put(-.5,0){\makebox(0,0)[cc]{\Large$\bullet$}}
\put(0,0){\makebox(0,0)[bl]{\scriptsize $Y$}}
\put(-.5,0){\vector(0,1){1.2}} \hskip 1em
\put(0.9,0.2){\makebox(0,0)[lb]{and}}
\end{picture}
\hskip 4.5em
\raisebox{-1em}{\rule{0pt}{0pt}}
\unitlength .4cm
\begin{picture}(0,1.4)(0,-.3)
\put(-.45,.55){\makebox(0,0)[cc]{$\bbox$}}
\put(-.5,-.8){\vector(0,1){1.2}}
\put(0.7,-.25){\makebox(0,0){,}}
\end{picture}
\end{equation}
and one ``planar'' binary vertex
\begin{equation}
\label{bbb}
\unitlength.5cm
\begin{picture}(5,1.5)(-2,.4)
\put(0,1.08){\vector(0,1){.85}}
\put(0,1){\makebox(0,0)[cc]{\Large$\bullet$}}
\put(0.3,1.2){\makebox(0,0)[l]{\scriptsize$F$}}
\put(-1,-1){
\put(0,1){{\vector(1,1){.92}}}
}
\put(1,-1){
\put(0,1){{\vector(-1,1){.92}}}
}
\end{picture}
\end{equation}
whose planarity means that its inputs are linearly 
ordered. In pictures, this order is determined by reading the inputs 
from left to right.

\section{The general case}
\label{s3}

Let us generalize calculations in
Section~\ref{s2} and describe
$\GLV$-invariant elements in 
\begin{equation}
\label{zabiraji_antibiotika?}
\Lin\left(\adj\Lin(\otexp V{{\hh}_1},\otexp V{p_1}) \ot \cdots \ot \Lin(\otexp
V{{\hh}_r},\otexp V{p_r}),\Lin(\otexp Vc,\otexp Vd)\right),
\end{equation}
where $r,\Rada p1r,\Rada {\hh}1r,c$ and $d$ are non-negative integers.
The above space is canonically isomorphic to
\[
\otexp{{V^*}}{p_1} \ot \otexp V{{\hh}_1} \ot \cdots \ot
\otexp{{V^*}}{p_r} \ot \otexp V{{\hh}_r} \ot \otexp{{V^*}}{c} \ot \otexp V{d},
\]
which is in turn isomorphic to\label{888}
\begin{equation}
\label{budu?}
\otexp {{V^*}}{(p_1 + \cdots + p_r + c)}
\ot \otexp V{({\hh}_1 + \cdots+ {\hh}_r + d)},
\end{equation}
via the isomorphism that moves all $V^*$-factors to the left, without
changing their relative order. By the last and first isomorphisms
in~(\ref{conon}), the space in~(\ref{budu?}) is isomorphic to
\[
\Lin(\otexp V{(p_1 + \cdots + p_r + c)},\otexp V{({\hh}_1 + \cdots+ {\hh}_r + d)}).
\]
We will denote the composite isomorphism
between~(\ref{zabiraji_antibiotika?}) and the space in the above
display by $\Phi$. 
Since all isomorphisms above are $\GLV$-equivariant, $\Phi$ is
equivariant, too,  thus the
space~(\ref{zabiraji_antibiotika?}) may  contain nontrivial 
$\GLV$-equivariant maps only if
\begin{equation}
\label{vyleci_mne_to?}
p_1 + \cdots + p_r + c = {\hh}_1 + \cdots + {\hh}_r + d.
\end{equation}

Denote by $\wGr$ the space spanned by all directed graphs with $r+1$ 
planar vertices
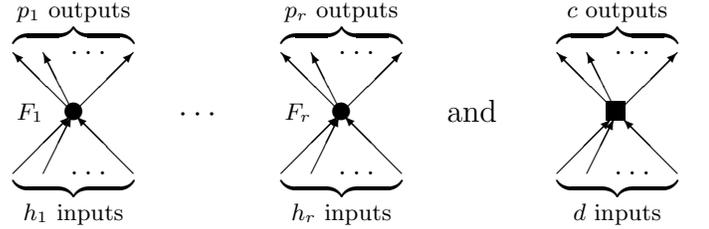
\[
\raisebox{-3.5em}{\rule{0pt}{0pt}}
\unitlength 4mm
\linethickness{0.4pt}
\begin{picture}(20,5.1)(10.5,19.4)
\put(20,20){\vector(1,1){2}}
\put(20,20){\vector(-1,1){2}}
\put(20,20){\vector(-1,2){1}}
\put(18,18){\vector(1,1){1.9}}
\put(22,18){\vector(-1,1){1.9}}
\put(19,18){\vector(1,2){.935}}
\put(20,20){\makebox(0,0)[cc]{\Large$\bullet$}}
\put(19,20){\makebox(0,0)[r]{\scriptsize $F_1$}}
\put(20.5,18){\makebox(0,0)[cc]{$\ldots$}}
\put(20,17){\makebox(0,0)[cc]{%
   $\underbrace{\rule{16mm}{0mm}}_{\mbox{\scriptsize ${\hh}_1$ inputs}}$}}
\put(0,40){
\put(20.5,-18){\makebox(0,0)[cc]{$\ldots$}}
\put(20,-17){\makebox(0,0)[cc]{%
   $\overbrace{\rule{16mm}{0mm}}^{\mbox{\scriptsize $p_1$ outputs}}$}}}
\end{picture}
\hskip -2.8cm
\raisebox{1.2mm}{$\cdots$}
\hskip -2.2cm
\begin{picture}(20,5.1)(10.5,19.4)
\put(20,20){\vector(1,1){2}}
\put(20,20){\vector(-1,1){2}}
\put(20,20){\vector(-1,2){1}}
\put(18,18){\vector(1,1){1.9}}
\put(22,18){\vector(-1,1){1.9}}
\put(19,18){\vector(1,2){.935}}
\put(20,20){\makebox(0,0)[cc]{\Large$\bullet$}}
\put(19,20){\makebox(0,0)[r]{\scriptsize $F_r$}}
\put(20.5,18){\makebox(0,0)[cc]{$\ldots$}}
\put(20,17){\makebox(0,0)[cc]{%
   $\underbrace{\rule{16mm}{0mm}}_{\mbox{\scriptsize ${\hh}_r$ inputs}}$}}
\put(0,40){
\put(20.5,-18){\makebox(0,0)[cc]{$\ldots$}}
\put(20,-17){\makebox(0,0)[cc]{%
   $\overbrace{\rule{16mm}{0mm}}^{\mbox{\scriptsize $p_r$ outputs}}$}}}
\end{picture}
\hskip -2.8cm
\raisebox{1.2mm}{\mbox{and}}
\hskip -2.2cm
\begin{picture}(20,5.1)(10.5,19.4)
\put(20,20){\vector(1,1){2}}
\put(20,20){\vector(-1,1){2}}
\put(20,20){\vector(-1,2){1}}
\put(18,18){\vector(1,1){1.8}}
\put(22,18){\vector(-1,1){1.8}}
\put(19,18){\vector(1,2){.9}}
\put(20,20){\makebox(0,0)[cc]{$\bbox$}}
\put(20.5,18){\makebox(0,0)[cc]{$\ldots$}}
\put(20,17){\makebox(0,0)[cc]{%
   $\underbrace{\rule{16mm}{0mm}}_{\mbox{\scriptsize $d$ inputs}}$}}
\put(0,40){
\put(20.5,-18){\makebox(0,0)[cc]{$\ldots$}}
\put(20,-17){\makebox(0,0)[cc]{%
   $\overbrace{\rule{16mm}{0mm}}^{\mbox{\scriptsize $c$ outputs}}$}}}
\end{picture} \hskip -3cm , \hskip 2cm
\]
where planarity means that linear orders of the sets of input and
output edges are specified. 
Observe that the number of edges of each graph spanning $\wGr$ equals
the common value of the sums in~(\ref{vyleci_mne_to?}).
For each graph $G \in \wGr$ we define a
$\GLV$-equivariant map $\wrR_n(G)$ in the
space~(\ref{zabiraji_antibiotika?}) as follows.

As in Section~\ref{s2},
choose a basis $(\Rada e1n)$ of $V$ and let
$(\rada{e^1}{e^n})$ be the corresponding dual basis of $V^*$. For
$F_i \in \Lin(\otexp V{{\hh}_i},\otexp V{p_i})$, $1 \leq i \leq r$, write
\[
F_i = 
{F_i \hskip .2em}^{a_1^i,\ldots,a^i_{p_i}}_{b_1^i,\ldots,b^i_{{\hh}_i}} \
e_{a_1} \ot \cdots \ot e_{a_{p_i}}
\otimes e^{b_1} \ot \cdots \ot e^{b_{{\hh}_i}}
\]
with some scalars 
${F_i \hskip .2em}^{a_1^i,\ldots,a^i_{p_i}}_{b_1^i,\ldots,b^i_{{\hh}_i}} \in \bfk$
or, more concisely,
$F_i = {F_i \hskip .2em}^{A^i}_{B^i}\  e_{A^i} \otimes e^{B^i}$,
where $A^i$ abbreviates the multiindex $(a_1^i,\ldots,a^i_{p_i})$,
$B^i$ the multiindex $(b_1^i,\ldots,b^i_{{\hh}_i})$, $e_{A^i} := e_{a_1} \ot
\cdots \ot e_{a_{p_i}}$, $e^{B^i} :=  e^{b_1} \ot \cdots \ot
e^{b_{{\hh}_i}}$ and, as everywhere in this paper, summations over
repeated (multi)indices are assumed. 

A {\em labelling\/} of a graph $G \in \wGr$ is a function $\ell :
\Edg(G) \to \{\rada 1n\}$, where $\Edg(G)$ denotes the set of edges of
$G$. Let $\Lab(G)$ be the set of all labellings of $G$.  For $\ell \in
\Lab(G)$ and $1 \leq i \leq r$, define $A^i(\ell)$ to be the multiindex
$(a_1^i,\ldots,a^i_{p_i})$ such that $a^i_s$ equals $\ell(e)$, where
$e$ is the edge that starts at the $s$-th output of the vertex $F_i$,
$1 \leq s \leq p_i$. Likewise, put $I(\ell) := (\Rada i1c)$ with $i_t
:= \ell(e)$, where now $e$ is the edge that starts at the $t$-th
output of the \raisebox{-.1em}{$\bbox$}-vertex, $1 \leq t \leq c$.
Let $B^i(\ell)$ and $J(\ell)$ have similar obvious meanings, with
`inputs' taken instead of `outputs.' For $F_1 \ot \cdots\ot F_r \in
\Lin(\otexp V{{\hh}_1},\otexp V{p_1}) \ot \cdots \ot \Lin(\otexp V{{\hh}_r},\otexp
V{p_r})$ define finally
\begin{equation}
\label{beru_antibiotika}
\wrR_n(G)(F_1 \ot \cdots\ot  F_r) :=
\sum_{\ell \in \Lab(G)} 
{F_1 \hskip .2em}^{A^1(\ell)}_{B^1(\ell)} \ot \cdots\ot 
{F_r \hskip .2em}^{A^r(\ell)}_{B^r(\ell)}\
 e_{J(\ell)} \ot e^{I(\ell)} \in \Lin(\otexp Vc,\otexp Vd). 
\end{equation}

It is easy to check that $\wrR_n(G)$ is a $\GLV$-fixed
element of the space~(\ref{zabiraji_antibiotika?}).
The nature of the summation in~(\ref{beru_antibiotika}) is close
to the {\em state sum model\/} for link invariants,
see~\cite[Section~I.8]{kauffman:KnotsandPhysics}, with states being
the values of labels of the edges of the graph.

\begin{proposition}
\label{zabere_to?}
Let $r,\Rada p1r,\Rada {\hh}1r,c$ and $d$ be non-negative integers. Then
the map 
\[
\wrR_n :\wGr \to \Lin_\GLV\left(\adj\Lin(\otexp V{{\hh}_1},\otexp V{p_1}) 
\ot \cdots \ot \Lin(\otexp
V{{\hh}_r},\otexp V{p_r}),\Lin(\otexp Vc,\otexp Vd)\right)
\] 
defined by~(\ref{beru_antibiotika}) is an epimorphism.  If $n
\geq e$, where $e$ is the number of edges of graphs spanning $\wGr$
and $n = \dim(V)$, $\wrR_n$ is also an isomorphism.
\end{proposition}

Observe that we do not need to assume~(\ref{vyleci_mne_to?}) in
Proposition~\ref{zabere_to?}.  If~(\ref{vyleci_mne_to?}) is not
satisfied, then there are no $\GLV$-invariant elements
in~(\ref{zabiraji_antibiotika?}) and also the space $\wGr$ is trivial,
thus $\wrR_n$ is an isomorphism of trivial spaces.

\begin{proof}[Proof of Proposition~\ref{zabere_to?}] By the above
observation, we may assume~(\ref{vyleci_mne_to?}). 
Consider the diagram
\begin{equation}
\label{nehoji_se_to}
\raisebox{-1.9cm}{\rule{0pt}{4cm}}
\unitlength 1cm
\linethickness{0.4pt}
\begin{picture}(15,1.2)(-.5,1.5)
\put(0,0){\makebox(0,0){$\wGr$}}
\put(8.6,0){\makebox(0,0){$\Lin_\GLV\left(\adj\Lin(\otexp V{{\hh}_1},\otexp V{p_1}) 
          \ot \cdots \ot \Lin(\otexp
          V{{\hh}_r},\otexp V{p_r}),\Lin(\otexp Vc,\otexp Vd)\right)$}}
\put(0,3){\makebox(0,0){$\bfk[\Sigma_k]$}}
\put(8.6,3){\makebox(0,0){$\Lin_\GLV(\otexp V{(p_1 + \cdots + p_r +
          c)},\otexp V{({\hh}_1 + \cdots+ {\hh}_r + d)})$}}
\put(.6,0){\vector(1,0){1.8}}
\put(.75,3){\vector(1,0){4}}
\put(0,.5){\vector(0,1){2}}
\put(8,.5){\vector(0,1){2}}
\put(8.3,1.5){\makebox(0,0)[l]{$\Phi$}}
\put(0.3,1.5){\makebox(0,0)[l]{$\Psi$}}
\put(7.7,1.5){\makebox(0,0)[r]{$\cong$}}
\put(-0.3,1.5){\makebox(0,0)[r]{$\cong$}}
\put(1.5,.15){\makebox(0,0)[b]{$\wrR_n$}}
\put(2.6,3.15){\makebox(0,0)[b]{${\mathcal R}_n$}}
\end{picture}
\end{equation}
in which ${\mathcal R}_n$ is the map~(\ref{preziji_to?}), $\wrR_n$ is
defined in~(\ref{beru_antibiotika}) and $\Phi$ is the composition of
canonical isomorphisms and reshufflings of factors described on
page~\pageref{888} above. The map $\Psi$ is defined as follows.

Let us denote, for the purposes of this proof only, by $\OUT(F_i)$ the
linearly ordered
set of outputs of the $F_i$-vertex, $1 \leq i \leq r$, 
and by $\OUT(\ctverecek)$ the linearly
ordered set of
outputs of \hskip 2pt  \raisebox{-1pt}{$\bbox$}. 
The set $\OUT := \OUT(F_1) \cup \cdots \cup
\OUT(F_r) \cup \OUT(\ctverecek)$ is linearly ordered  by requiring that
\[
\OUT(F_1) < \cdots < \OUT(F_r) < \OUT(\ctverecek)
\]
(we believe that the meaning of this shorthand is obvious). Let $\IN$
be the linearly ordered set of inputs defined in the similar way. The
orders define unique isomorphisms
\begin{equation}
\label{mam_chripku}
\OUT \cong (\rada 1k) \ \mbox { and } \IN \cong (\rada 1k)
\end{equation}
 of ordered sets.
 
Since graphs spanning $\wGr$ are determined by specifying how the
outputs of vertices are connected to its inputs, there exists a
one-to-one correspondence $G \leftrightarrow \varphi_G$ between graphs
$G \in \wGr$ and isomorphisms $\varphi_G : \OUT \stackrel{\cong}{\to}
\IN$. Given~(\ref{mam_chripku}), such $\varphi_G$ can be interpreted
as an element of the symmetric group $\Sigma_k$. The map $\Psi$ is
then defined by $\Psi(G) := \varphi_G$.

It is simple to verify that the diagram~(\ref{nehoji_se_to}) commutes,
so the proposition follows from the Invariant Tensor Theorem.
\end{proof}

\section{Symmetries occur}
\label{s4}

In the light of diagram~(\ref{nehoji_se_to}),
Proposition~\ref{zabere_to?} may look just as a clumsy reformulation of
the Invariant Tensor Theorem.  Graphs become relevant when
symmetries occur.

\begin{example}
\label{456}
Let $\Sym(\otexp V2,V) \subset \Lin(\otexp V2,V)$ be the subspace of
symmetric bilinear maps, i.e.~maps satisfying $f(v',v'') =
f(v'',v')$ for $v',v'' \in V$.  Let us explain how to use
calculations of Section~\ref{s2} to describe
$\GLV$-equivariant maps in $\Lin\left(\adj\otexp V2 \sqot \Sym(\otexp
V2,V),V\right)$.
 
The right $\Sigma_2$-action on $\Lin(\otexp V2,V)$ given by permuting
the inputs of bilinear maps is such that the space $\Sym(\otexp V2,V)$
equals the subspace $\Lin(\otexp V2,V)^{\Sigma_2}$ of $\Sigma_2$-fixed
elements.  This right $\Sigma_2$-action induces a left
$\Sigma_2$-action on $\Lin\left(\adj\otexp V2 \sqot \Lin(\otexp
V2,V),V\right)$ which commutes with the $\GLV$-action, therefore it
restricts to a left $\Sigma_2$-action on the subspace
$\Lin_\GLV\left(\adj\otexp V2 \sqot \Lin(\otexp V2,V),V\right)$ of
$\GLV$-equivariant maps.

There is also a left $\Sigma_2$-action on the linear space $\wGr_{\rm ex}$
interchanging the inputs of the $F$-vertices of generating graphs. It
is simple to check that the map~(\ref{boli_mne_v_krku}) of
Section~\ref{s2} is equivariant with respect
to these two $\Sigma_2$-actions, hence it induces the map
\begin{equation}
\label{porad_mi_neni_dobre}
\Sigma_2  \backslash \wrR_n :\Sigma_2 \backslash \wGr_{\rm ex} 
\to \Sigma_2  \backslash 
\Lin_\GLV\left(\adj\otexp V2 \ot  \Lin(\otexp V2,V),V\right)
\end{equation}
of left cosets. Observe that, by a standard duality argument,
\begin{equation}
\label{ttt}
\Sigma_2 \backslash\Lin_\GLV\left(\adj\otexp V2 \ot 
\Lin(\otexp V2,V),V\right)\cong 
\Lin_\GLV\left(\adj\otexp V2 \ot \Sym(\otexp V2,V),V\right).
\end{equation}

Let us denote $\wGr_{{\rm ex},\bullet} := \Sigma_2 \backslash\wGr_{\rm
ex}$.  The bullet $\bullet$ in the subscript signalizes the presence
of vertices with fully symmetric inputs.  By definition, graphs $G',
G'' \in \wGr_{\rm ex}$ are identified in the quotient 
$\wGr_{{\rm ex},\bullet}$ if they
differ only by the order of inputs of the $F$-vertex. In
Figure~\ref{table}, this identification is indicated by vertical
braces.  We see that $\wGr_{{\rm ex},\bullet}$ 
is again a space {\em spanned by
graphs,\/} this time with no linear order on the inputs of the
$F$-vertex. So we may {\em define\/} $\wGr_{{\rm ex},\bullet}$ as the space
spanned by directed graphs with vertices~(\ref{aaa}) and one binary
(ordinary, non-planar) vertex~(\ref{bbb}). We conclude by
interpreting~(\ref{porad_mi_neni_dobre}) as the map
\begin{equation}
\label{Phillips}
\wrR_n : \wGr_{{\rm ex},\bullet} \to \Lin_\GLV\left(\adj\otexp V2 \ot  
\Sym(\otexp V2,V),V\right).
\end{equation}
It follows from the properties of the map~(\ref{boli_mne_v_krku}) and
the characteristic zero assumption that $\wrR_n$ is always an
epimorphism and is an isomorphism if $n \geq 3$.
\end{example}

At this point we want to incorporate, by generalizing the pattern 
used in Example~\ref{456}, symmetries into Proposition~\ref{zabere_to?}.
Unfortunately, it turns out that treating the
space~(\ref{zabiraji_antibiotika?}) in full generality leads to a
notational disaster.
To keep the length of formulas within a
reasonable limit, we decided to {\em assume from now on\/} that 
$p_1= \cdots = p_r = 1$,
$c=0$ and $d=1$. This means that we will restrict our attention to 
maps in 
\begin{equation}
\label{red}
\Lin\left(\adj\Lin(\otexp V{{\hh}_1},V) \ot \cdots \ot \Lin(\otexp
V{{\hh}_r},V),V\right).
\end{equation}
For graphs this assumption implies that the vertices $\Rada F1r$
have precisely one output,  and that the anchor
$\hskip .2em\raisebox{-.1em}{\bbox}\hskip -.2em$ 
has one input and no outputs. The number of inputs of $F_i$ will be
called the {\em arity\/} of $F_i$, $1 \leq i \leq r$.
Condition~(\ref{vyleci_mne_to?}) reduces to
\[
r = {\hh}_1 + \cdots + {\hh}_r + 1
\]
and one also sees that $r$ equals the number of edges of the
generating graphs.

The above generality is sufficient for all applications we have in mind.
A modification to the general case is straightforward but 
notationally challenging.

The space $\Lin(\otexp V{\hh},V)$ admits, for each ${\hh} \geq 0$, a
natural right $\Sigma_{\hh}$-action given by permuting inputs of
multilinear maps.  
A {\em symmetry\/} of maps in $\Lin(\otexp V{\hh},V)$ 
will be specified by a subset  $\fin \subset
\bfk[\Sigma_{\hh}]$. We then denote
\[
\Lin_\fin(\otexp V{\hh},V) := \left\{\adj f \in  \Lin(\otexp V{\hh},V)
;\ f {\mathfrak s} = 0 \mbox {
  for each } {\mathfrak s} \in \fin\right\}. 
\]
For $\fin$ as above and a left $\Sigma_{\hh}$-module $U$, we will abbreviate
by $\fin \backslash U$ the left coset   $\fin U \backslash U$.

\begin{example}
\label{exxx}
Let $\fin := I_{\hh} \subset \bfk[\Sigma_\hh]$ be the augmentation ideal.
Then $\Lin_{I_{\hh}}(\otexp V{\hh},V)$ is the space of symmetric maps,
\[
\Lin_{I_{\hh}}(\otexp V{\hh},V) = \Sym(\otexp V{\hh},V),
\]
therefore the augmentation ideal describes the symmetry of the local 
coordinates of vector fields and their derivatives, 
see~\cite[Example~3.2]{markl:na}.
We leave as an exercise to describe in this language the spaces of
{\em anti\/}symmetric maps.
\end{example}

\begin{example}
\label{exxy}
Let ${\hh} := v+2$, $v \geq 0$, and let $\nabla \subset \bfk[\Sigma_{\hh}]$ be
the image of the augmentation ideal $I_v$ of $\bfk[\Sigma_v]$ in
$\bfk[\Sigma_{\hh}]$ under the map of group rings induced by the inclusion
$\Sigma_v \hookrightarrow \Sigma_v \times \Sigma_2
\hookrightarrow\Sigma_{\hh}$ that interprets permutations of $(\rada 1v)$
as permutations of $(\rada 1v,v+1,v+2)$ keeping the last two elements
fixed. Then $\Lin_\nabla(\otexp V{\hh},V)$ consists of multilinear maps
$\otexp V{(v+2)} \to V$ that are symmetric in the first $v$ inputs,
i.e.~multilinear maps possessing the symmetry of the Christoffel
symbols of linear connections and their derivatives, see
again~\cite[Example~3.2]{markl:na}. 
\end{example}

\begin{remark}
\label{patek_v_IHES}
It is clear how to generalize the above notion of symmetry to maps in
the left $\Sigma_p$- right $\Sigma_{\hh}$-module $\Lin(\otexp V{\hh},\otexp
Vp)$ for general $p,{\hh} \geq 0$. A symmetry of these maps
will be specified by
subsets $\Si \in \bfk[\Sigma_{{\hh}}]$ and $\So \in \bfk[\Sigma_{p}]$,
the corresponding subspaces will then be
\[
\Lin_{\Si}^{\So}(\otexp {V}{{\hh}},\otexp {V}{p}) :=
\left\{\adj f \in  \Lin(\otexp V{\hh},\otexp Vp)
;\ f {\mathfrak s} = 0 = {\mathfrak t} f \mbox {
for each } {\mathfrak s} \in \Si \mbox { and } {\mathfrak t} \in \So
\right\}.
\]
\end{remark}

Suppose we are given subsets $\fin_i \subset \bfk[\Sigma_{{\hh}_i}]$, $1
\leq i \leq r$.  Our aim is to describe $\GLV$-invariant elements in
the space
\begin{equation}
\label{reds}
\Lin\left(\adj\Lin_{\fin_1}(\otexp V{{\hh}_1},V) 
\ot \cdots \ot \Lin_{\fin_r}(\otexp
V{{\hh}_r},V),V\right).
\end{equation}
Let 
\[
\fin :=\fin_1 \cup \cdots \cup \fin_r
\subset \bfk[\Sigma_{{\hh}_1} \times \cdots \times \Sigma_{{\hh}_r}],
\]
where $\fin_i$ is, for $1 \leq i \leq r$, identified with
its image in $\bfk[\Sigma_{{\hh}_1} \times \cdots \times \Sigma_{{\hh}_r}]$
under the map induced by the group inclusion $\Sigma_{{\hh}_i} \hookrightarrow
\Sigma_{{\hh}_1} \times \cdots \times \Sigma_{{\hh}_r}$. 

As in Example~\ref{456}, we use the fact that, for $1 \leq i \leq r$,
each $\Lin(\otexp V{{\hh}_i},V)$ is a right $\Sigma_{{\hh}_i}$-space, hence
the tensor product $\Lin(\otexp V{{\hh}_1},V) \ot \cdots \ot \Lin(\otexp
V{{\hh}_r},V)$ has a natural right $\Sigma_{{\hh}_1} \times \cdots \times
\Sigma_{{\hh}_r}$-action which induces a left $\Sigma_{{\hh}_1} \times \cdots
\times \Sigma_{{\hh}_r}$-action on the space~(\ref{red}). This action
restricts to the subspace of $\GLV$-equivariant maps.

There is also a left $\Sigma_{{\hh}_1} \times \cdots \times
\Sigma_{{\hh}_r}$-action on the space $\wGr$ given by permuting, in the
obvious manner, the inputs of the vertices $\Rada F1r$ of generating
graphs. The map $\wrR_n$ of Proposition~\ref{zabere_to?} is equivariant
with respect to the above two actions and induces the map
\[
\fin \backslash \wrR_n : \fin \backslash \wGr
\to 
\fin \backslash 
\Lin_\GLV\left(\adj\Lin(\otexp V{{\hh}_1},V) \ot \cdots \ot \Lin(\otexp
V{{\hh}_r},V),V\right)
\]
of left quotients. Denoting  $\wGr_\fin := \fin \backslash \wGr$ 
and realizing that, by duality, the codomain of $\fin \backslash \wrR_n$ is 
isomorphic to the subspace of $\GLV$-fixed
elements in~(\ref{reds}), we obtain the map (denoted again $\wrR_n$)
\begin{equation}
\label{jeste_jeden_den}
\wrR_n : \wGr_\fin \to   \Lin_\GLV\left(\adj\Lin_{\fin_1}(\otexp V{{\hh}_1},V) 
\ot \cdots \ot \Lin_{\fin_r}(\otexp
V{{\hh}_r},V),V\right)
\end{equation}
which is, by Proposition~\ref{zabere_to?}, 
an epimorphism and is an isomorphism if $\dim(V) \geq r$.

\begin{remark}
\label{jaja}
As in Example~\ref{456}, it turns out that 
the quotient $\wGr_\fin = \fin \backslash \wGr$ is a 
{\em space of graphs\/} though, for general symmetries,
``space of graphs'' means a free wheeled operad on a certain
$\Sigma$-module~\cite{mms}. 
In the cases relevant for our paper, we however remain in the realm of
`classical' graphs, as shown in the following example, see also the
proof of Corollary~\ref{boli_mne_za_krkem}.
\end{remark}

\begin{example}
\label{ja}
Suppose that, for some $1 \leq i \leq r$, $\fin_i$ equals the
augmentation ideal $I_{{\hh}_i}$ of $\bfk[\Sigma_{{\hh}_i}]$ as in 
Example~\ref{exxx}.
Then, in the quotient $\fin \backslash \wGr$, one identifies graphs that
differ by the order of inputs of the vertex $F_i$. In other words,
modding out by $\fin_i \subset \fin$ erases the order of inputs of
$F_i$, turning $F_i$ into an ordinary (non-planar) vertex.  If $\fin_i =
\nabla$ as in Example~\ref{exxy}, one gets a
vertex of arity $v+2$, $v \geq 0$, 
whose first $v$ inputs are symmetric.
\end{example}

For applications, we 
still need one more level of generalization that will reflect the
antisymmetry of the Chevalley-Eilenberg
complex~\cite[Section~2]{markl:na} in the Lie algebra variables. As
a motivation for our construction, we offer the following continuation
of the calculations in Section~\ref{s2} and Example~\ref{456}.

\begin{example}
\label{zivotosprava}
We will consider the tensor product $V \ot V$ as a left
$\Sigma_2$-module, with the action $\tau(v' \ot v'') := -
(v'' \ot v')$, for $v',v'' \in V$ and the generator $\tau \in
\Sigma_2$. The subspace $(V \ot V)^{\Sigma_2}$ of $\Sigma_2$-fixed
elements is then precisely the second exterior power $\ext^2 V$.  This left
action induces a $\GLV$-equivariant right $\Sigma_2$-action on the
space $\Lin\left(\adj\otexp V2 \sqot \Sym(\otexp V2,V),V\right)$ such that
\[
\Lin\left(\adj\otexp V2 \ot \Sym(\otexp V2,V),V\right)/\Sigma_2 \cong 
\Lin\left(\adj\ext^2 V \ot \Sym(\otexp V2,V),V\right).
\]
The above isomorphism restricts to an isomorphism
\begin{equation}
\label{u}
\Lin_\GLV\left(\adj\otexp V2 \ot \Sym(\otexp V2,V),V\right)/\Sigma_2 \cong 
\Lin_\GLV\left(\adj\ext^2 V \ot \Sym(\otexp V2,V),V\right). 
\end{equation}
of the subspaces of $\GLV$-equivariant maps.

Likewise, $\wGr_{{\rm ex},\bullet}$ carries a right $\Sigma_2$-action that
interchanges the labels $X$ and $Y$ of the \black-vertices of graphs
in the last column of Figure~\ref{table} and multiplies the sign of
the corresponding generator by $-1$. The map~(\ref{Phillips}) is
$\Sigma_2$-equivariant, therefore it induces the map
\[
\wrR_n/\Sigma_2 : \wGr_{{\rm ex},\bullet} /\Sigma_2 \to
\Lin_\GLV\left(\adj\otexp V2 \ot \Sym(\otexp V2,V),V\right)/\Sigma_2.
\]
Let us denote $\Gr^2_{{\rm ex},\bullet} 
:= \wGr_{{\rm ex},\bullet}/\Sigma_2$ and $\rR^2_n :=
\wrR_n/\Sigma_2$. Using~(\ref{u}), one rewrites the above map as an
epimorphism
\[
\rR^2_n : \Gr^2_{{\rm ex},\bullet} 
\epi \Lin_\GLV\left(\adj\ext^2 V \sqot \Sym(\otexp V2,V),V\right)
\] 
which is an isomorphism if $n \geq 3$.

The space $\Gr^2_{{\rm ex},\bullet}$ is isomorphic to the span of the set of
directed, oriented graphs with one (non-planar) binary vertex $F$, an
anchor \anchor, and two `white' vertices \white.  By an {\em
orientation\/} we mean a linear order of white vertices. A graph with
the opposite orientation is identified with the original one taken with the
opposite sign. It is clear that, with $\Gr^2_{{\rm ex},\bullet}$ 
defined in this way, the
map $\Gr^2_{{\rm ex},\bullet} 
\to \wGr_{{\rm ex},\bullet} /\Sigma_2$ that replaces the first
(in the linear order given by the orientation) white vertex \white\ by
the black vertex \black\ labelled by $X$, and the second white vertex
by the black vertex labelled by $Y$, is an isomorphism.

The symmetry of the inputs of the vertex $F$ implies 
the following identities in $\Gr^2_{{\rm ex},\bullet}$:
\[
\unitlength.5cm
\begin{picture}(10,2)(0,.5)
\put(0,2){\makebox(0,0)[cc]{\hskip .5mm$\bbox$}}
\put(0,1.08){\vector(0,1){.85}}
\put(0,1){\makebox(0,0)[cc]{\Large$\bullet$}}
\put(0.3,1.2){\makebox(0,0)[l]{\scriptsize$F$}}
\put(-1,-1){
\put(0.15,1.15){{\vector(1,1){.76}}}
\put(0,1){\makebox(0,0)[cc]{\Large$\circ$}}
\put(1,1){\makebox(0,0)[cc]{$<$}}
}
\put(1,-1){
\put(-.15,1.15){{\vector(-1,1){.76}}}
\put(0,1){\makebox(0,0)[cc]{\Large$\circ$}}
}
\put(2.5,1){\makebox(0,0){$=-$}}
\put(5,0){
\put(0,2){\makebox(0,0)[cc]{\hskip .5mm$\bbox$}}
\put(0,1.08){\vector(0,1){.85}}
\put(0,1){\makebox(0,0)[cc]{\Large$\bullet$}}
\put(0.3,1.2){\makebox(0,0)[l]{\scriptsize$F$}}
\put(-1,-1){
\put(0.15,1.15){{\vector(1,1){.76}}}
\put(0,1){\makebox(0,0)[cc]{\Large$\circ$}}
\put(1,1){\makebox(0,0)[cc]{$>$}}
}
\put(1,-1){
\put(-.15,1.15){{\vector(-1,1){.76}}}
\put(0,1){\makebox(0,0)[cc]{\Large$\circ$}}
}
\put(2.5,1){\makebox(0,0){$=-$}}
}
\put(10,0){
\put(0,2){\makebox(0,0)[cc]{\hskip .5mm$\bbox$}}
\put(0,1.08){\vector(0,1){.85}}
\put(0,1){\makebox(0,0)[cc]{\Large$\bullet$}}
\put(0.3,1.2){\makebox(0,0)[l]{\scriptsize$F$}}
\put(-1,-1){
\put(0.15,1.15){{\vector(1,1){.76}}}
\put(0,1){\makebox(0,0)[cc]{\Large$\circ$}}
\put(1,1){\makebox(0,0)[cc]{$<$}}
}
\put(1,-1){
\put(-.15,1.15){{\vector(-1,1){.76}}}
\put(0,1){\makebox(0,0)[cc]{\Large$\circ$}}
}
}
\put(11.5,1){\makebox(0,0){,}}
\end{picture}
\adjust {1.2}
\]
from which one concludes that
\[
\hskip  5cm 
\unitlength.5cm
\begin{picture}(10,2)(0,.5)
\put(0,2){\makebox(0,0)[cc]{\hskip .5mm$\bbox$}}
\put(0,1.08){\vector(0,1){.85}}
\put(0,1){\makebox(0,0)[cc]{\Large$\bullet$}}
\put(0.3,1.2){\makebox(0,0)[l]{\scriptsize$F$}}
\put(-1,-1){
\put(0.15,1.15){{\vector(1,1){.76}}}
\put(0,1){\makebox(0,0)[cc]{\Large$\circ$}}
\put(1,1){\makebox(0,0)[cc]{$<$}}
}
\put(1,-1){
\put(-.15,1.15){{\vector(-1,1){.76}}}
\put(0,1){\makebox(0,0)[cc]{\Large$\circ$}}
}
\put(2.5,1){\makebox(0,0){$=0$.}}
\end{picture}
\adjust {1.2}
\]
Therefore $\Gr^2_{{\rm ex},\bullet}$ 
is in this case one-dimensional, spanned by the
equivalence class of the oriented directed graph
\[
\begin{picture}(7,5)(1,5)
\unitlength.74cm
\put(-1,-1.5){
\put(0,-.1){
\put(0,2){\put(0.03,0){\makebox(0,0)[cc]{$\bbox$}}}
\put(0,1.17){\vector(0,1){.75}}
\put(0,1){\makebox(0,0)[cc]{\Large$\circ$}}
}
\put(.4,.2){
\put(2.09,.85){\makebox(0,0)[cc]{\oval(1.5,1.5)[b]}}
\put(2.09,1.15){\makebox(0,0)[cc]{\oval(1.5,1.5)[t]}}
\put(2.85,1.22){\line(0,1){.3}}
\put(.82,.82){\vector(1,1){.48}}
\put(.7,.7){\makebox(0,0)[cc]{\Large$\circ$}}
\put(1.35,1.35){\makebox(0,0)[cc]{\Large$\bullet$}}
\put(1.32,1.25){\makebox(0,0)[tc]{\vector(0,1){0}}}
\put(1,1.45){\makebox(0,0)[r]{\scriptsize $F$}}
\put(0.16,0.7){\makebox(0,0)[cc]{$<$}}
\put(3.4,1){\makebox(0,0)[b]{.}}
}}
\end{picture}
\adjust {2}
\] 
In the notation of Figure~\ref{table}, the above graph represents the
map that sends $(X\land Y) \ot F \in \ext^2 V \ot \Sym(\otexp V2,V)$ into
\[
X \otimes \Tr(F(Y,-)) - Y \otimes \Tr(F(X,-)) \in V.
\]
\end{example}

Let us turn to our final task. We want to describe $\GLV$-invariant
elements in the space
\begin{equation}
\label{piano-nad-hlavou}
\Lin\left(\Ext_{1 \leq i \leq m} \Sym(\otexp V{{\hh}_i},V) \ot
\bigotimes_{m+1 \leq i \leq r} \Lin_{\fin_i}(\otexp V{{\hh}_i},V),V\right)
\end{equation}
where, as before, $r,\Rada {\hh}1r$ are positive integers, $\fin_i
\subset \bfk[\Sigma_{{\hh}_i}]$ for $m+ 1 \leq i \leq r$, and $m$ is
an integer such that $1 \leq m \leq r$. Having in mind the description
of the space of symmetric multilinear maps given in
Example~\ref{exxx}, we extend the definition of $\fin_i$ also to 
$1 \leq i \leq m$,  by putting $\fin_i: = I_{\hh_i}$.  
The first step is to identify
the exterior power $\Land_{1 \leq i \leq m} \Sym(\otexp V{{\hh}_i},V)$
with the fixed point set of an action of a suitable finite group. This
can be done as follows.

For $1 \leq w \leq m$, let $A(w) \subset \{\rada 1m\}$ be the subset
$A(w) := \{1 \leq i \leq m;\ {\hh}_i = {\hh}_w\}$. Then
\[
\{\rada 1m\} = \textstyle\bigcup_{1 \leq w \leq m} A(w)
\] 
is a decomposition of $\{\rada 1m\}$ into not necessarily distinct
subsets.  Let $\fA \subset \Sigma_m$ be the subgroup of permutations
of $\{\rada 1m\}$ preserving this decomposition.

The group $\fA$ acts on $\bigotimes _{1 \leq i \leq m}
\Sym(\otexp V{{\hh}_i},V)$ by permuting the corresponding
factors. If we consider this tensor product as a left $\fA$-module
with this permutation action twisted by the signum representation, then
\[
\Ext_{1 \leq i \leq m} \Sym(\otexp V{{\hh}_i},V) \cong
\left(\bigotimes_{1 \leq i \leq m} \Sym(\otexp V{{\hh}_i},V)
\right)^\fA.
\]
The above left $\fA$-action on $\bigotimes _{1 \leq i
\leq m} \Sym(\otexp V{{\hh}_i},V)$ induces a dual $\GLV$-equivariant 
right $\fA$-action on the space~(\ref{piano-nad-hlavou}).

There is a right $\fA$-action on the quotient $\wGr_\fin = \fin
\backslash \wGr$ defined as
follows. For a graph $G \in \wGr$ representing an element $[G] \in
\wGr_\fin$ and for $\sigma \in \fA$, let
$G^\sigma$ be the graph obtained from $G$ by permuting
the vertices $\Rada F1m$ according to $\sigma$. We then put $[G]\sigma
:= {\rm sgn\/}(\sigma)  [G^\sigma]$. Since, by the
definition of $\fA$, $\sigma$ may interchange only vertices with the
same number of inputs and the same symmetry, our definition of
$G^\sigma$ makes sense. 

It is simple to see that 
the map $\wrR_n$ in~(\ref{jeste_jeden_den}) is $\fA$-equivariant, 
giving rise to the map
\[
\wrR_n / \fA : \wGr_\fin/\fA \to 
\Lin_\GLV(\Lin_{\fin_1}(\otexp V{{\hh}_1},V) \ot \cdots \ot
\Lin_{\fin_r}(\otexp V{{\hh}_r},V),V)/\fA
\]
of right cosets. 
The codomain of $\wrR_n/ \fA$ is easily seen to be isomorphic to the
subspace of $\GLV$-equivariant elements
in~(\ref{piano-nad-hlavou}). The above calculations are summarized in the
following proposition in which 
$\Gr^m_\fin := \wGr_\fin/\fA$ and $\rR^m_n:= \wrR_n / \fA$. 

\begin{proposition}
\label{zabere_to??}
Let $r,\Rada \hh 1r$ be non-negative integers, $1 \leq m \leq r$, and 
$\fin_i \subset \bfk[\Sigma_{\hh_i}]$ for $m+1 \leq i \leq r$. Then the map
\begin{equation}
\label{zitra_na_kole}
\rR^m_n :\Gr^m_\fin \to \Lin_\GLV\left(\Ext_{1 \leq i \leq m} 
\Sym(\otexp V{\hh_i},V)  \ot
\bigotimes_{m+1 \leq i \leq r} 
\Lin_{\fin_i}(\otexp V{\hh_i},V),V\right)
\end{equation}
constructed above is an epimorphism. 
If, moreover, the dimension $n$ of $V$ 
$\geq$ the number of edges of graphs spanning
$\Gr^m_\fin$, $\rR^m_n$ is also an isomorphism.
\end{proposition}

The following result says that the presence of vertices with symmetric
inputs miraculously {\em extends\/} the stability range
(Definition~\ref{stab}). In applications, these vertices will
represent the Lie algebra generators in the Chevalley-Eilenberg
complex.

\begin{proposition}
\label{zitra_odletam_z_IHES_do_Prahy}
Suppose that $\Rada \hh 1m \geq 2$.
If $n \geq e-m$, where $n$ is the dimension of $V$ and $e$ 
the number of edges of graphs spanning
$\Gr^m_\fin$, then the map $\rR^m_n$ in Proposition~\ref{zabere_to??} 
is an isomorphism.
\end{proposition}

\begin{proof}

Let $G$ be a 
graph spanning $\Gr^m_\fin$ and $S \subset \Edg(G)$ a subset of edges of $G$ 
such that $\card(S) > n$. For each permutation $\sigma$ of elements
of $S$, denote by $G_\sigma$ the graph obtained by cutting the edges
belonging to $S$ in the middle and regluing them following the
automorphism $\sigma$. The linear combination
\begin{equation}
\label{eeee}
\sum_{\sigma \in \Sigma_S}{\rm sgn\/}(\sigma) \cdot  G_\sigma \in
\Gr^m_\fin
\end{equation}
is then a graph-ical representation of the expression
in~(\ref{Pozitri_zpet_do_Prahy}), thus the kernel of $\rR^m_n$ is
generated by expressions of this type.  Since, by assumption,
$\card(S) \leq n+m$ and $\Rada \hh 1m \geq 2$, the set $S$ must
necessarily contain two input edges of the {\em same\/} symmetric
vertex of $G$.  This implies that the sum~(\ref{eeee}) vanishes,
because with each graph $G_\sigma$ it contains the same graph with the
opposite sign.  This shows that the kernel of $\rR^m_n$ is trivial.
\end{proof}

\begin{remark}
\label{po_navratu_z_polska}
By an absolutely straightforward generalization of the above
constructions, one can obtain versions of
Proposition~\ref{zabere_to??} and
Proposition~\ref{zitra_odletam_z_IHES_do_Prahy} describing the space
\begin{equation}
\label{Eli_chce_prijet_v_prosinci.}
\Lin_\GLV\left(\Ext_{1 \leq i \leq m} 
\Sym(\otexp V{\hh_i},V)  \ot
\bigotimes_{m+1 \leq i \leq r} 
\Lin_{\fin_i}^{\So_i}(\otexp V{\hh_i},\otexp V{p_i}), 
\Lin_{\fin}^{\So}(\otexp V{c},\otexp V{d})\right)
\end{equation}
in terms of a space spanned by graphs. Since the notational aspects of
such a generalization are horrendous, we must leave the details as an
exercise to the reader.
\end{remark}

\section{A particular case}
\label{s5}

We finish this note by a corollary tailored for the needs of~\cite{markl:na}.
For non-negative integers $m,b$ and $c$, denote by
$\Gr^m_{\bullet(b)\nabla(c)}$ the space spanned by directed, oriented
graphs with
\begin{itemize} 
\item[(i)]
$m$ unlabeled `white' vertices with fully symmetric inputs and 
arities $\geq 2$,
\item[(ii)] 
$b$ `black' labelled vertices with fully
symmetric inputs and arities $\geq 0$,
\item[(iii)]
$c$ labelled $\nabla$-vertices, and 
\item[(iv)]
the anchor \anchor.
\end{itemize}

In item~(iii), a $\nabla$-vertex means a vertex with the symmetry
described in Example~\ref{exxy}, see also Example~\ref{ja}. 
As in Example~\ref{zivotosprava}, an {\em orientation\/}
is given by 
a linear order on the set of white vertices. If $G'$ and $G''$ are graphs
in $\Gr^m_{\bullet(b)\nabla(c)}$ whose orientations differ by an odd number of
transpositions, then we identify $G' = -G''$ in
$\Gr^m_{\bullet(b)\nabla(c)}$.

\begin{corollary}
\label{boli_mne_za_krkem}
For each non-negative integers $m,b$ and $c$ there exists a natural
epimorphism
\begin{eqnarray*}
\lefteqn{
\rR^m_{\bullet(b)\nabla(c),n} :\Gr^m_{\bullet(b)\nabla(c)} \epi}
\\ 
&& \bigoplus_{\vec h \in \frH}
\Lin_\GLV \hskip -.2em
\left(\Ext_{1 \leq i \leq m} \hskip -.2em
\Sym(\otexp V{{\hh}_i},V)  \ot
\bigotimes_{m+1 \leq i \leq m+b} \hskip -1.2em  \Sym(\otexp V{{\hh}_i},V) 
\bigotimes_{m+b+1 \leq i \leq m+b+c} \hskip -1.2em 
\Lin_\Delta(\otexp V{{\hh}_i},V),V\right),
\end{eqnarray*}
with the direct sum taken over the set $\frH$ of all multiindices
$\vec h = (\Rada h1{m+b+c})$ such that  
\[
\Rada h1m \geq 2,\ 
\Rada h{m+1}{m+b} \geq 0\  \mbox { and }\  \Rada h{m+b+1}{m+b+c} \geq 2.
\]  
The map $\rR^m_{\bullet(b)\nabla(c),n}$ is an isomorphism if $n = \dim(V) \geq
b+c$.
\end{corollary}

\begin{proof}
The map $\rR^m_{\bullet(b)\nabla(c),n}$ is constructed by assembling
the maps $\rR^m_n$ from Proposition~\ref{zabere_to??} as follows. For
a multiindex $\vec h = (\Rada h1{m+b+c}) \in \frH$ as in the corollary 
take, in  Proposition~\ref{zabere_to??},  $r:= m+b+c$ and
\[
\fin_i = \fin_i(\vec h) := 
\cases{\adj I_{{\hh}_i}}{for $m+1 \leq i \leq m+b$ and}%
{\rule{0pt}{1.2em}\nabla}{for $m+b+1 \leq i \leq r$,}
\] 
see Examples~\ref{exxx}
and~\ref{exxy} for the notation. Let $\rR^m_n(\vec h)$ be the
map~(\ref{zitra_na_kole}) corresponding to the above choices and
$\rR^m_{\bullet(b)\nabla(c),n} := 
\bigoplus_{\vec h \in \frH}\rR^m_n(\vec h)$. 
We only need to show that the graph
space $\Gr^m_{\bullet(b),\nabla(c)}$ is isomorphic to the direct
sum of the double quotients
$\Gr^m_{\fin(\vec h)} = \fin(\vec h) \backslash \wGr /\fA$.

As we argued in Example~\ref{ja}, the left quotient 
$\wGr_{\fin(\vec h)}  = {\fin(\vec h)}  \backslash \wGr$ 
is spanned by directed graphs with $r$ labelled
vertices $\Rada F1r$ such that the 1st type vertices $\Rada F1m$
(`white' vertices) have
fully symmetric inputs and arities ${\hh}_1,\ldots, {\hh}_m$, and the remaining
vertices $\Rada F{m+1}r$ are as in items (ii)--(iv) of the definition of
$\Gr^m_{\bullet(b)\nabla(c)}$ but with fixed
arities $\rada {h_{m+1}}{h_r}$.

Modding out $\wGr_{\fin(\vec h)}$  by $\fA$ identifies graphs 
that differ by a relabelling of white vertices of
the same arity and the sign given by to the signum of this relabelling.
This clearly means that the map 
\[
\Gr^m_{\bullet(b),\nabla(c)} \to 
\bigoplus_{\vec h \in \frH} \Gr^m_{\fin(\vec h)} = \bigoplus_{\vec h \in \frH}
\wGr_{\fin(\vec h)} / \fA
\] 
that assigns to the first
(in the linear order given by the orientation) white vertex of graphs
generating $\Gr^m_{\bullet(b),\nabla(c)}$
label $F_1$, to the second white vertex label $F_2$, etc., 
is an isomorphism. By simple combinatorics, graphs spanning
$\Gr^m_{\bullet(b),\nabla(c)}$ have precisely $m+b+c$ edges which 
completes the proof of the corollary.
\end{proof}

\begin{remark}
\label{ZASE_mne_boli_v_krku}
Proposition~\ref{zabere_to??} and its Corollary~\ref{boli_mne_za_krkem} was
obtained by applying the double-coset reduction $\fin \backslash \ {-}
/\fA$ and standard duality to the map $\wrR_n$ of
Proposition~\ref{zabere_to?}. 
Backtracking all the constructions involved, one can see that, in
Corollary~\ref{boli_mne_za_krkem}, the invariant
linear map $\rR^m_{\bullet(b)\nabla(c),n}(G)$ corresponding to a graph $G \in
\Gr^m_{\bullet(b)\nabla(c)}$ is given by 
the `state sum'~(\ref{beru_antibiotika}) {\em antisymmetrized\/} 
in the white vertices.
\end{remark}


\def\cprime{$'$}

\end{document}